\documentclass[12pt]{amsart}
\usepackage{fullpage}
\usepackage{latexsym}
\usepackage{amssymb}
\usepackage{amssymb,amsmath,epsfig,graphics,latexsym,psfrag}
\usepackage[english]{babel}
\usepackage[all]{xy}
\usepackage{amscd, amssymb, amsmath, amsthm, graphics}
\usepackage{amsmath,amsfonts,amsthm,amssymb}
\usepackage{latexsym,amsmath}
\usepackage{graphicx,psfrag}
\usepackage{mathrsfs}

\newtheorem{thmx}{Theorem}

\newtheorem{thm}{Theorem}[section]
\newtheorem{prop}[thm]{Proposition}
\newtheorem{cor}[thm]{Corollary}

\newtheorem{lem}[thm]{Lemma}

\usepackage{hyperref}

\newcommand{\Z}{\mathbb{Z}}

\newcommand{\SO}{\mathrm{\SO}}

\title[Mapping degrees  between $S^1$-bundles]
{ Maps  between circle bundles: 
 Fiber-preserving, Finiteness and Realization of mapping degree sets}

\author[C. Neofytidis]{Christoforos Neofytidis}
\address{Department of Mathematics and Statistics, University of Cyprus, Nicosia 1678, Cyprus}
\email{neofytidis.christoforos@ucy.ac.cy}
\author[H. Sun]{Hongbin Sun}
\address{Department of Mathematics, Rutgers University - New Brunswick, Hill Center, Busch Campus, Piscataway, NJ 08854, USA}
\email{hongbin.sun@rutgers.edu}
\author[Y. Tian]{Ye Tian}
\address{Morningside Center of Mathematics, 
Academy of Mathematics and System Science, 
Chinese Academy of Sciences,
Beijing 100190, China}
\email{ytian@math.ac.cn}
\author[S. Wang]{Shicheng Wang}
\address{Department of Mathematical Sciences, Peking University, Beijing 100871, China}
\email{wangsc@math.pku.edu.cn}
\author[Z. Wang]{Zhongzi Wang}
\address{Department of Mathematical Sciences, Peking University, Beijing 100871, China}
\email{wangzz22@stu.pku.edu.cn}

\date{\today}
\subjclass[2010]{55M25}
\keywords{Mapping degrees, circle bundles, fiber-preserving map, Euler class, finiteness and realization problems}

\begin{document}

\maketitle

\begin{abstract} Let $E_i$ be an oriented circle bundle over a closed  oriented $n$-manifold $M_i$ with $i=1,2$ and assume that $M_2$ is aspherical.  
 We prove the following geometric realization result, which extends the classical result for maps between circle bundles over surfaces to all dimensions $\geq 2$. 
 Let   $ f\colon   E_1\to  E_2$ be a map.
If the induced homomorphism $f_*$ on $\pi_1$ sends the fiber subgroup of $\pi_1(E_1)$ to the fiber subgroup of $\pi_1(E_2)$, then there is a fiber-preserving map in the homotopy class of $f$.

Using 
the above result and a computation of the mapping degree set of fiber-preserving maps between $E_1$ and $E_2$, 
we obtain the following results on the finiteness and the realization problems of mapping degree sets.
 \begin{itemize}
\item[($\mathcal F$)]  The mapping degree set  $D(E_1, E_2)$ is finite if $M_2$ is hyperbolic and the Euler class of $E_2$ is not torsion.
\item[($\mathcal R$)] For any finite set $A$ of integers containing $0$ and any $n>2$, $A$ is the mapping degree set $D(M,N)$
for some closed oriented $n$-manifolds 
$M$ and $N$. 
\end{itemize}

Item ($\mathcal F$) extends the previously known $3$-dimensional case to all dimensions $\geq 3$. 
Item ($\mathcal R$) gives a complete answer to the realization problem on finite sets of integers (containing $0$) in all dimensions, including in particular the previously unknown cases in dimensions $ 4$ and $5$.
\end{abstract}

\tableofcontents
\section{Introduction}
Let $M,N$ be two closed, connected,  oriented manifolds  of the same dimension.
The degree of a map $f\colon M\to N$, denoted by $\deg(f)$, is 
probably one of the  oldest and most fundamental concepts in topology.
The {\em set of mapping degrees} from $M$ to $N$ is defined by
\[
D(M,N):=\{d\in\Z \ | \ \exists \ f\colon M\to N, \ \deg(f)=d\}.
\]
 When $M=N$, the {\em set of degrees of self-maps} $D(M,M)$ is denoted by $D(M)$.

 The study of various properties of mapping degree sets for various classes of $M$ and $N$, including 
the finiteness of $D(M,N)$,  whether $\pm 1\in D(M,N)$,  as well as the exact computation of $D(M,N)$, has a long history. This topic was revolutionized by Thurston~\cite{Th} and Gromov~\cite{Gr} more than 40 years ago, and the most prominent notion that arose from these profound works was the {\em simplicial volume}. Thereafter, research on these questions became very active, rich in open problems and applications, and involves various methods from geometric topology, algebraic topology, 
differential geometry, representation theory, analysis, and others. For a (highly non-exhaustive) list of references featuring the variety of these techniques, we refer the reader to  
\cite{BG}, \cite{KL},
 \cite{LS}, \cite{BRW},
 \cite{DLSW}, \cite{CMV},  \cite{Ne}, 
\cite{BGM}, 
and the bibliographies therein. 
   
In this paper, we mainly study maps between circle bundles of the same dimension and their mapping degree sets. In the following, we will describe our results and their applications.

\subsection{Fiber-preserving maps} $S^1$-bundles are primary examples of fiber bundles
and constitute a significant class of manifolds. In particular, among the eight classes of geometric 3-manifolds in the sense of Thurston, six of them are finitely covered by $S^1$-bundles over surfaces, which are called Seifert manifolds. The following result is classical in 3-manifold topology, which is important in the study of mapping class groups 
and non-zero degree maps of 3-manifolds.

\begin{thmx}\label{Seifert} If $M$ and $N$ are circle bundles over aspherical surfaces  and $f_*\colon \pi_1(M)\to \pi_1(N)$ sends the fiber subgroup to the fiber subgroup, then there is a fiber-preserving map in the homotopy class of $f$.
\end{thmx}

The proof of Theorem \ref{Seifert} appeared in Waldhausen's 1967 paper \cite{Wal}  for self-maps, 
 and in Rong's 1993 paper \cite{Ro} for the general case. Both proofs use the hierarchy theory of 2- and 3-manifolds.
 
When the surfaces in Theorem \ref{Seifert} are replaced by $n$-manifolds for $n>2$, such a theorem was not known and 
 our first goal is to establish such a result.

\begin{thm}\label{main2}
Let  $E_i$ be an oriented $S^1$-bundle over a closed oriented $n$-manifold $M_i$ with $i=1,2$, and assume that $M_2$ is aspherical. Let $ f\colon   E_1\to  E_2$ be a map.
If the induced map $f_*$ on $\pi_1$ sends the fiber subgroup of $\pi_1(E_1)$ to the fiber subgroup of $\pi_1(E_2)$, then there is a fiber-preserving map in the homotopy class of $f$.
\end{thm}

Here the fiber subgroup of $\pi_1(E_i)$ is the normal subgroup given by the image of $\pi_1(S^1)\to \pi_1(E_i)$.

\noindent{\bf Remark 1.}
The proofs of  Theorems \ref{Seifert} and \ref{main2} are different. The proof of Theorem \ref{Seifert} uses the hierarchy theory of  2- and  3-manifolds, and such a hierarchy theory is not available in higher dimensions. 
The proof of Theorem \ref{main2}  uses the theory of $S^1$-bundles, in particular that oriented $S^1$-bundles over $M$ can be regarded as principal bundles, whose isomorphism classes are in 1-1 correspondence with $H^2(M, \Z)$, and the gauge class group of an $S^1$-bundle over $M$ is isomorphic to $\text{Hom}(\pi_1(M), \Z)=H^1(M, \Z)$.

Our next goal is to 
describe the   mapping degree sets of fiber-preserving maps between $S^1$-bundles.
Given two closed oriented circle bundles $E_1$ and $E_2$ of the same dimension, we define the {\em fiber-preserving  mapping degree set} as
\[
D_{FP}(E_1,E_2):=\{d\in\Z \ |\, \text{$\exists$ fiber-preserving map} \  f\colon E_1\to E_2, \ \text{such that}\ \deg( f)=d\}.
\]

\begin{thm}\label{main1} Let  $M_i$ be a closed oriented $n$-manifold and let 
$E_i \to M_i$ be an oriented $S^1$-bundle with Euler class $e_i\in H^2(M_i;\mathbb{Z})$, $i=1,2$. Then 
$$D_{FP}(E_1, E_2)=\{0\} \cup\{k\cdot \deg(f)\ |\  k\ne 0,  f\colon M_1\to M_2, \deg(f)\ne 0, \text{and}\ f^\#(e_2)=ke_1\},$$
where $f^\#\colon H^2(M_2;\mathbb{Z})\to H^2(M_1;\mathbb{Z})$ is  the induced homomorphism on the second cohomology.
\end{thm}

\subsection{Applications}

\subsubsection{Finiteness of mapping degree sets}\label{intro:f}
A primary question about a mapping degree set $D(M, N)$ is whether it is finite.  We say a numerical invariant $v$ of $n$-manifolds is a domination invariant if for any map $f\colon M\to N$, we have $v(M)\ge |\deg(f)|\cdot v(N)$.

Clearly $D(M, N)$  is finite  for any $M$ if $v(N)$ is positive and finite for some domination invariant $v$.
 The simplicial volume is  the most important domination invariant. It is positive for hyperbolic manifolds, but it vanishes for circle bundles.
 On the other hand, for  an $S^1$-bundle $E$ over a hyperbolic surface with non-zero Euler class, $w(E)>0$ for the Seifert volume $w$,
which is a domination invariant introduced by Brooks and Goldman~\cite{BG}, thus $D(M, E)$ is finite for any $3$-manifold $M$.
  
 So far there is no known non-zero and finite domination invariant for $S^1$-bundles over $n$-manifolds, with $n\ge 3$~\cite{Ne}.    
The first significant consequence of Theorems \ref{main2} and \ref{main1} is on the finiteness of mapping degree sets between
$S^1$-bundles. For brevity, 
  we will often use $\tilde M_a$ to denote the oriented circle bundle over a closed oriented $n$-manifold $M$ with Euler class $e(\tilde M_a)=a\in H^2(M;\mathbb{Z})$. 
Our finiteness result is the following one.

\begin{thm}[Finiteness Theorem]\label{appl1}
 Suppose $N$ is a closed oriented  hyperbolic $n$-manifold and  $b\in H^2(N;\mathbb{Z})$ is  not torsion. 
 Then $D(\tilde M_a,\tilde N_b)$ is a finite set for any  closed oriented $n$-manifold $M$ and any $a\in H^2(M;\mathbb{Z})$.
\end{thm}

\noindent{\bf Remark 2.} Theorem \ref{appl1} implies that  the self-mapping degree set  $D(\tilde N_b)\subset \{0, \pm 1\}.$
 This fact about $D(\tilde N_b)$  was proved in  \cite[Theorem 1.3]{Ne}.  
 A direct argument shows that a non-zero degree map $f\colon \tilde M_a\to \tilde N_b$ 
sends the fiber subgroup to the fiber subgroup (see the proof of Corollary \ref{fiber3}, or \cite{JS}, \cite{Sc}, \cite{Ro}, \cite{Ne}).
 Starting from this point, the proof of Theorem \ref{appl1} and that of \cite[Theorem 1.3]{Ne} are different.
The proof in \cite{Ne}
 uses a group theoretic idea developed in~\cite{NeIIPP} and some dynamical argument, since self-maps can be iterated. 
The maps in Theorem \ref{appl1} cannot be iterated, and the proof of  Theorem \ref{appl1} relies on  Theorems \ref{main2} and \ref{main1}.

\subsubsection{The realization problem}

The problem of realizing an arbitrary set of integers (containing zero) as a mapping degree set appeared   in the literature rather recently, although it has been present over the last decades, at least implicitly. It was shown in \cite[Theorem 1.3]{NWW} that there exist uncountably many infinite sets $A\subseteq\Z$ with $0\in A$ that are not equal to $D(M,N)$, for any closed oriented $n$-manifolds $M,N$. So the following refined version of the realization problem arose.

\vskip 0.2 truecm
  \noindent{\bf Problem 1.}{ \cite[Problem 1.4]{NWW}
 Suppose $A$ is a finite set of integers containing $0$. Are there closed oriented $n$-manifolds $M$ and $N$ such that $A=D(M,N)$?
\vskip 0.2 truecm

Our interest in $S^1$-bundles also stems from the fact that $S^1$-bundles over surfaces are instrumental for answering Problem 1.
At first, given a non-zero integer $k$, the set $\{0, k\}$, the smallest non-zero subset of $\mathbb{Z}$ containing 0, is realized by 
$D(\tilde \Sigma_1, \tilde \Sigma_k)$, see~\cite[Lemma 3.4]{NWW}. 
Here $\tilde \Sigma_k$ is the $S^1$-bundle over a closed oriented hyperbolic surface $\Sigma$ with Euler number $k$.
Then by using connected sums of non-trivial $S^1$-bundles over hyperbolic surfaces, and their products, many finite subsets of $\mathbb{Z}$ are
realized as mapping degree sets \cite[Theorems 1.7 and 1.9]{NWW}. Soon after these results, 
 C. Costoya, V. Mu\~noz, and A. Viruel gave a complete positive answer to Problem 1 in a stronger form, by only using connected sums of $S^1$-bundles over surfaces, and some arithmetic combinatorics~\cite[Prop. 2.2]{CMV}.
 
\begin{thmx}\label{t:CMV} \cite[Theorem A]{CMV}.
If $A$ is a finite set of integers containing $0$, then $A = D(M,N)$
for some closed oriented connected $3$-manifolds $M,N$.
\end{thmx}

Combining Theorem \ref{t:CMV} and the fact that  every mapping degree set of $n$-manifolds is the mapping degree set of some $(n+k)$-manifolds for any $k\ge 3$ \cite[Theorem 2.2]{NSTWW}, we proved the following result.

\begin{thmx}\label{extension}  \cite[Theorem 2.3]{NSTWW}
For each positive integer $n\ne 1, 2, 4, 5$,
every finite set of integers containing $0$ is the mapping degree set 
of a pair of  $n$-manifolds. 
\end{thmx}

Theorem \ref{extension} obviously fails for $n=1,2$.
One motivation of the present work is to show that 
Theorem \ref{extension} also holds for $n=4$ and $5$, therefore giving a complete solution to Problem 1 in all dimensions. Indeed, as we shall explain below, our study of maps between $S^1$-bundles yields this consequence. We first prove the following result.

\smallskip

\begin{thm}\label{0,k}
For each non-zero integer $k$ and  $n=4, 5$,  there exist closed oriented aspherical  $n$-manifolds of the form $\tilde N_b,\tilde N_{kb}$ such that  
$$D(\tilde N_b,\tilde N_{kb})=\{0, k\}.$$
 \end{thm}

By applying the arithmetic combinatorial construction of \cite[Prop 2.2]{CMV}, Theorems \ref{0,k} and \ref{extension} imply the following result.

\begin{thm}[Realization Theorem]\label{realization}
For every finite set $A$ of integers containing 0 and  each integer $n>2$, there exist closed, oriented $n$-manifolds 
$M$ and $N$ such that $D(M,N)=A$.
\end{thm}

{\bf Acknowledgments.} An anonymous referee told us that some conditions posed on Theorem \ref{main2} of the original version could be removed, as well as how to remove them. We thank the referee for the insightful advice.

C.N. is partially supported by a start-up fund from the University of Cyprus, as well as by the research grant EXCELLENCE/0925-{\em Hopf problem and monotone invariants} from the Research and Innovation Foundation (RIF) of Cyprus.
H.S. is partially supported by Simons Collaboration Grants 615229.  Y.T. and  S.W. are partially supported by NSFC grant No. 12571071. Z.W. is partially supported by NSFC grant No. 125B2006.}

\section{Fiber-preserving: from group level to topology level}

The goal of this section is to prove Theorem \ref{main2}. In this section, CW-complexes will be either finite or countable.  
We begin by recalling some facts on $S^1$-bundles, see \cite{MS}, \cite{Ha}, \cite{Mo}.

Let  $M_i$ be a CW-complex, and let 
$S^1\overset{j_i}\to E_i\ \overset{p_i}\to M_i$ be an oriented $S^1$-bundle, $i=1,2$. 
We say that a map  $\tilde  f\colon E_1\to E_2$ is {\em fiber-preserving} if it maps each $S^1$-fiber of $E_1$ to an $S^1$-fiber of $E_2$. In this case, $\tilde{f}$ induces a map $ f\colon M_1\to M_2$. 
A fiber-preserving map $ \tilde{f}\colon E_1\to E_2$ is a {\em bundle map} if the restriction of $\tilde{f}$ on each $S^1$ fiber is a homeomorphism.
Suppose now $M_1=M_2=M$.  We say that a fiber-preserving map $\tilde f\colon E_1\to E_2$ is {\em vertical} if it induces the identity map on  $M$.
Finally, two oriented $S^1$-bundles $E_1 \overset{p_1}\longrightarrow M$ and $E_2 \overset{p_2}\longrightarrow M$ are isomorphic if there is a bundle map $\tilde f\colon E_1\to E_2$ such that $p_1=p_2 \circ \tilde f$ and $\tilde{f}$ preserves orientations of all $S^1$-fibers.

For each map $f\colon M\to N$ and an $S^1$-bundle $\tilde N$ over $N$, we have the pull-back bundle $ {f^*(\tilde N})$ over $M$ and a bundle map $\tilde f\colon f^*(\tilde N) \to  \tilde N$ such that the following diagram commutes.
\begin{align}\label{d1}
\begin{CD}
			 f^*(\tilde N) @>  \tilde f>> \tilde N  \\
			@VV p_1 V   @VV p_2 V  \\
		       M @> f >> N
\end{CD} 
\end{align}
Some basic facts about $S^1$-bundles are stated in the following.

\begin{thm}\label{basic} \
\begin{enumerate} 
\item[(i)] Any oriented $S^1$-bundle $E$ over $M$ is isomorphic to a pull-back bundle of the universal $S^1$-bundle 
$S^{\infty}\to \mathbb{C}P^{\infty}$, i.e., there is a map $f\colon M\to   \mathbb{C}P^{\infty}$ such that the following diagram commutes.
\begin{align}\label{2.2}
\begin{CD}
			 E \cong f^*(S^{\infty})@> \tilde f >> S^{\infty}  \\
			@VV p_1 V   @VV p V  \\
		       M @> f >>  \mathbb{C}P^{\infty}
\end{CD}   
\end{align}
Moreover, $f$ is unique up to homotopy. 
Therefore, there is a one-to-one  correspondence between isomorphism classes of oriented $S^1$-bundles  over $M$ 	and elements of $[M, \mathbb{C}P^\infty]$,
the set of homotopy classes of maps  $M\to \mathbb{C}P^{\infty}$.

\item[(ii)] Since $\mathbb{C}P^{\infty}$ is a model of the Eilenberg-MacLane space $K(\mathbb{Z}, 2)$, we have 
$$[M, \mathbb{C}P^\infty]\cong [M,  K(\mathbb{Z}, 2)]\cong H^2(M;\mathbb{Z}).$$
 Therefore, there 
 is a one-to-one  correspondence between isomorphism classes of oriented circle bundles $E$ over $M$ 
and elements in $H^2(M;\mathbb{Z})$,  which sends $E$ to its Euler class $e(E)\in H^2(M;\mathbb{Z})$.

\item[(iii)] Since the embedding of $\text{SO}(2)\subset \text{Diff}_0(S^1)$, respectively $\text{SO}(2)\subset \text{Homeo}_0(S^1)$,   is a homotopy equivalence, there is a one-to-one   correspondence between isomorphism classes of oriented $S^1$-bundles over $M$
and isomorphism classes of principal $\text{SO}(2)$-bundles over $M$ in the smooth category, respectively in the topological category.
(As usual, by $\text{Diff}_0(M)$, respectively $\text{Homeo}_0(M)$, we denote the component of the diffeomorphism group, respectively the homeomorphism group, of $M$ containing the identity.)  
\end{enumerate}
\end{thm}

In the rest of the paper, we omit the term ``isomorphism classes". 

Actually, we prove a stronger version of Theorem \ref{main2} for CW-complexes, which is stated below.

\begin{thm}\label{main3.1'}
Let 
$S^1\overset{j_i}\to E_i\ \overset{p_i}\to M_i$ be an oriented $S^1$-bundle over a CW-complex $M_i$ with $i=1,2$, and assume that $M_2$
is aspherical. Let $\tilde f\colon   E_1\to  E_2$ be a map.
If the induced map $\tilde f_*$ on $\pi_1$ sends the fiber subgroup to the fiber subgroup, then there is a fiber-preserving map in the homotopy class of $\tilde f$.
\end{thm}

Before proving Theorem \ref{main3.1'}, we first show some preliminary results.

\begin{lem}\label{algebra}
Let $1\to Z\to G\xrightarrow{p} \Gamma \to 1$ and $1\to Z'\to G'\xrightarrow{p'} \Gamma' \to 1$ be two short exact sequences of groups given by central extensions. Let $i\colon Z\to Z'$ be an isomorphism, and let $\psi\colon\Gamma\to \Gamma'$, $\phi_j\colon G\to G'$ be homomorphisms such that the following diagram commutes for $j=1,2$. 
\begin{align}\label{3.0}
\begin{CD}
		1@> >> Z @> >>	G @>  p>> \Gamma @ > >> 1 \\
		@ VVV	@VVi  V  @VV \phi_j V   @VV \psi V  @ VV V \\
		1@> >>  Z' @>  >>      G' @> p' >> \Gamma' @ > >> 1.
\end{CD}
\end{align}
Then there is a self-homomorphism $\kappa\colon G\to G$ such that $\phi_1\circ \kappa=\phi_2$. Moreover, there exists a homomorphism $\lambda\colon\Gamma\to Z$, such that $\kappa(g)=\lambda(p(g))\cdot g$ for any $g\in G$.
\end{lem}

\begin{proof}
Here, we naturally consider $Z$ as a subgroup of $G$, and $Z'$ as a subgroup of $G'$. So we have $\phi_j|_Z=i$ for $j=1,2$. We define $\kappa:G\to G$ by 
$$\kappa(g)=i^{-1}(\phi_2(g)\cdot \phi_1(g)^{-1})\cdot g.$$
Note that $\phi_2(g)\cdot \phi_1(g)^{-1}\in Z'$ holds since 
$$p'(\phi_2(g)\cdot \phi_1(g)^{-1})=p'(\phi_2(g))\cdot p'(\phi_1(g))^{-1}=\psi(p(g))\cdot \psi(p(g))^{-1}=e.$$

We check that $\kappa$ is indeed a group homomorphism:
\begin{align*}
& \kappa(g_1)\cdot \kappa(g_2)=i^{-1}(\phi_2(g_1)\cdot \phi_1(g_1)^{-1})\cdot g_1\cdot i^{-1}(\phi_2(g_2)\cdot \phi_1(g_2)^{-1})\cdot g_2\\
=\ & i^{-1}\big(\phi_2(g_1)\cdot \phi_1(g_1)^{-1}\cdot \phi_2(g_2)\cdot \phi_1(g_2)^{-1}\big)\cdot g_1\cdot g_2\\
=\ & i^{-1}\big(\phi_2(g_1)\cdot \phi_2(g_2)\cdot \phi_1(g_2)^{-1}\cdot \phi_1(g_1)^{-1}\big)\cdot (g_1\cdot g_2)\\
=\ & i^{-1}\big(\phi_2(g_1\cdot g_2)\cdot \phi_1(g_1\cdot g_2)^{-1}\big)\cdot (g_1\cdot g_2)=\kappa(g_1\cdot g_2).
\end{align*}
Here, the second and third equalities hold since $Z$ and $Z'$ lie in the center of $G$ and $G'$, respectively.

We check that $\phi_1\circ \kappa=\phi_2$:
\begin{align*}
& \phi_1\circ \kappa(g)=\phi_1\big(i^{-1}(\phi_2(g)\cdot \phi_1(g)^{-1})\cdot g\big)=\phi_1\big(i^{-1}(\phi_2(g)\cdot \phi_1(g)^{-1})\big)\cdot \phi_1(g)\\
=\ & (\phi_2(g)\cdot \phi_1(g)^{-1})\cdot \phi_1(g)=\phi_2(g).
\end{align*}

For the moreover part, we first define $\tilde{\lambda}:G\to Z$ by $$\tilde{\lambda}(g)=i^{-1}(\phi_2(g)\cdot \phi_1(g)^{-1}).$$ We can check that $\tilde{\lambda}$ is a homomorphism, by a similar computation that proves $\kappa$ is a homomorphism. For any $z\in Z$, we have $\tilde{\lambda}(z)=i^{-1}(\phi_2(z)\cdot \phi_1(z)^{-1})=i^{-1}(i(z)\cdot i(z^{-1}))=e$. So $Z<\text{ker}\tilde{\lambda}$, and $\tilde{\lambda}$ induces a homomorphism $\lambda:\Gamma\to Z$ such that $\tilde{\lambda}=\lambda\circ p$. By definitions of $\kappa$ and $\tilde{\lambda}$, we have $\kappa(g)=\tilde{\lambda}(g)\cdot g=\lambda(p(g))\cdot g$.
\end{proof}

In fact, $\kappa\colon G\to G$ is a group isomorphism, but we will not need this fact in our proof of Theorem \ref{main3.1'}. 

Let us put the above lemma in the context of circle bundles.  Suppose $S^1\overset{j}\to E\ \overset{p}\to M$ is an oriented $S^1$-bundle over a  CW-complex  $M$ so that $\pi_1(S^1)$ injects into $\pi_1(E)$ (e.g. when $M$ is aspherical). Given a homomorphism $\phi\colon\pi_1(E)\to \pi_1(E)$ that sends the fiber subgroup $\pi_1(S^1)<\pi_1(E)$ to itself, the restriction of $\phi$ gives a homomorphism $\phi|\colon\pi_1(S^1)\to \pi_1(S^1)$, and $\phi$ induces a homomorphism $\bar{\phi}\colon\pi_1(M)\to \pi_1(M)$.

\begin{prop}\label{selfmap}
Suppose $\phi\colon \pi_1(E)\to \pi_1(E)$ is an isomorphism that sends the fiber subgroup to itself, such that $\phi|_{\pi_1(S^1)}=id_{\pi_1(S^1)}$ and $\bar{\phi}=id_{\pi_1(M)}$. 
Then there exists a bundle isomorphism $f_\phi\colon E\to E$ that induces $\phi$ on $\pi_1$.
\end{prop}

\begin{proof}
We apply Lemma \ref{algebra} to the following commutative diagram, with $\phi_1=id$ and $\phi_2=\phi$.
\begin{align*}
\begin{CD}
		1@> >> \pi_1(S^1) @> >>	\pi_1(E) @>  p_*>> \pi_1(M) @ > >> 1 \\
		@.	@VVid  V  @VV \phi_j V   @VV id  V  @ . \\
		1@> >>  \pi_1(S^1) @>  >>  \pi_1(E) @> p_* >> \pi_1(M) @ > >> 1.
\end{CD}
\end{align*}
Then there is a homomorphism $\lambda\colon\pi_1(M)\to \pi_1(S^1)$, such that $\phi(g)=\lambda(p_*(g))\cdot g$.

By Theorem \ref{basic} (iii), we can assume that $E$ is a principal $\text{SO}(2)$-bundle. Since $\text{SO}(2)\cong S^1$ as Lie groups, we have an $S^1$-action on $E$. We do not distinguish between left and right actions since $S^1$ is abelian.

Since $S^1$ is a $K(\pi,1)$-space, there exists $f\colon M\to S^1$ realizing $\lambda\colon\pi_1(M)\to \pi_1(S^1)$.  Without loss of generality, we assume the basepoint $m_0\in M$ is mapped to $1\in S^1$, and we identify $S^1$ with the fiber $p^{-1}(m_0)\subset E$ over $m_0$.  We define $f_\phi\colon E\to E$ by
$$f_\phi(x)=f(p(x))\bullet x,$$
where $\bullet$ denotes the $S^1$-action. Then $f_\phi\colon E\to E$ is clearly a bundle isomorphism.

It suffices to show that $f_\phi$ induces $\phi$ on $\pi_1$. For $f\colon[0,1]\to S^1$ and $g\colon[0,1]\to E$, we define $f\bullet g\colon[0,1]\to E$ by $(f\bullet g)(t)=f(t)\bullet g(t)$. We take the basepoint $e_0\in p^{-1}(m_0)\subset E$ that corresponds to $1\in S^1$ under the identification between $S^1$ and $p^{-1}(m_0)$. For any based loop $\gamma\colon[0,1]\to E$ such that $\gamma(0)=\gamma(1)=e_0$, we need to check that $(f_{\phi})_*([\gamma])=\phi([\gamma])\in \pi_1(E,e_0)$. Here, $(f_{\phi})_*([\gamma])$ is represented by $f_{\phi}\circ \gamma\colon[0,1]\to E$, with 
$$f_{\phi}\circ \gamma(t)=f(p(\gamma(t)))\bullet \gamma(t)=\big((f\circ p\circ \gamma)\bullet \gamma\big) (t).$$ 
Also, we have 
$$\phi([\gamma])=\lambda(p_*([\gamma]))\cdot [\gamma]=f_*(p_*([\gamma]))\cdot [\gamma]=[f\circ p \circ \gamma]\cdot [\gamma]=[(f\circ p \circ \gamma)\cdot \gamma].$$
Here $f\circ p\circ \gamma$ denotes the path in $S^1\cong p^{-1}(m_0)\subset E$, and the $\cdot$ in $(f\circ p \circ \gamma)\cdot \gamma$ denotes the concatenation of paths.
So $\phi([\gamma])$ is represented by 
$$(f\circ p \circ \gamma)\cdot \gamma=\big((f\circ p\circ \gamma)\cdot c_1\big)\bullet \big(c_{e_0}\cdot \gamma\big)\simeq (f\circ p\circ \gamma)\bullet \gamma.$$
Here $c_1$ and $c_{e_0}$ denote the constant maps to $1\in S^1$ and $e_0\in E$ respectively, and $\simeq$ denotes homotopy relative to $\{0,1\}$. 

So we verified that $(f_{\phi})_*([\gamma])=\phi([\gamma])$ for any $[\gamma]\in \pi_1(E,e_0)$, thus $(f_{\phi})_*=\phi$ holds.
\end{proof}

\begin{lem} \label{1} 
Suppose $S^1\overset{j_i}\to E_i\ \overset{p_i}\to M$ are oriented $S^1$-bundles, $i=1,2$, where $M$ is  aspherical. Suppose $\phi\colon\pi_1(E_1)\to \pi_1(E_2)$ is an isomorphism such that the following diagram commutes.
$$\begin{CD}
	      1@> >> \pi_1(S^1) @>  >> \pi_1( E_1) @>  (p_1)_*>> \pi_1(M) @ > >> 1   \\
		@.	@VVid  V @ VV \phi V   @VV id V @. \\
		1@> >>  \pi_1(S^1) @>  >> \pi_1( E_2)@> (p_2)_* >> \pi_1(  M) @ > >> 1.
\end{CD}$$
Then $E_1$ and $E_2$ are isomorphic as  oriented $S^1$-bundles.
\end{lem}

\begin{proof} This is a known fact.  First  the Euler classes of these two $\mathbb{Z}$-central extensions are equal to each other in $H^2(\pi_1(M);\mathbb{Z})$, see  for example \cite[pages 234-235]{FS}. Since $M$ is aspherical,
 we have a natural isomorphism $H^2(\pi_1(M);\mathbb{Z})\to H^2(M;\mathbb{Z})$, and the image of the Euler class of each $\mathbb{Z}$-central extension is the Euler class of the $S^1$-bundle. Then Theorem \ref{basic} (ii) implies that $E_1$ is isomorphic to $E_2$.
 \end{proof}
 
 By homotopy theory, for each CW-complex $M$, the $K(\pi_1(M),1)$-space $X$ can be constructed as a colimit of an infinite sequence of CW-complexes $X^{(n)}$. Here $X^{(2)}=M$ and for each $n\ge 3$, $X^{(n)}$ is obtained by attaching cells of dimension $n$ to $X^{(n-1)}$ annihilating $\pi_{n-1}(X^{(n-1)})$.  We may assume that $K(\pi_1(M),1)$ is constructed by the above method, so we have the following decomposition $K(\pi_1(M),1)=M\cup \bigcup_{i=3}^{+\infty} \bigcup_{\alpha \in J_i} e^i_\alpha$, where $J_i$ is an index set for all $i\ge 3$ and $e_\alpha^i$ is an $i$-dimensional cell for each $\alpha\in J_i$. 
 So we have the canonical embedding $i_M: M\to X=K(\pi_1(M), 1)$ that induces the identity on $\pi_1$. Under  this setting,
the following lemma and its proof are suggested by the referee.
 
 \begin{lem} \label{2} Let $E$ be an oriented $S^1$-bundle over a CW-complex $M$ such that  the fiber inclusion map $S^1\to E$ is $\pi_1$-injective. Then $E$ can be extended to an oriented $S^1$-bundle $\tilde E$ over $K(\pi_1(M), 1)$ 
 such that  $\tilde E$ is a $K(\pi_1(E), 1)$ space, and
the inclusion $i_E: E \to \tilde E$ induces the identity homomorphism on $\pi_1$. \end{lem}
 
 \begin{proof}  
 By the above discussion, we may assume that $K(\pi_1(M),1)=M\cup \bigcup_{i=3}^{+\infty} \bigcup_{\alpha \in J_i} e^i_\alpha$, where $J_i$ is an index set for all $i\ge 3$ and $e_\alpha^i$ is an $i$-dimensional cell for each $\alpha\in J_i$, and we have that $X^{(n)}=M\cup \bigcup_{i=3}^n \bigcup_{\alpha \in J_i} e^i_\alpha$. 
 We will extend the $S^1$-bundle $E$ over $M$
 inductively to $X^{(3)}$,  $X^{(4)}$ and so on to get $\tilde E$. 
 
Once the extension is done, in each step of the extension, the $\pi_1$ is unchanged, so the inclusion $i_E: E \to \tilde E$ induces the identity homomorphism on $\pi_1$. Moreover, as an $S^1$-bundle over an aspherical space,
 $\tilde E$ is an aspherical space, therefore a $K(\pi_1(E), 1)$ space.
Now we prove the extension exists.

 {\bf Claim:}  The classifying map $f_E: M\to K(\mathbb Z, 2)$ of $E$ (recall Theorem \ref{basic}) induces the trivial 
 homomorphism on $\pi_2$.
 
 \begin{proof} If not, there exists a $g: S^2\to M$ such that $f_E\circ g: S^2\to K(\mathbb Z, 2)$ is not null-homotopic.
 Then the pullback bundle $g^*E\to S^2$ is non-trivial. It follows that $g^*E$ is a lens space and therefore the inclusion map 
 $S^1\to g^*E$ is not $\pi_1$-injective, and therefore $S^1\to g^*E\to E$ is not $\pi_1$-injective, which contradicts our assumption.
 \end{proof}
 
 By the Claim, $f_E:M\to K(\mathbb{Z},2)$ can be extended to $X^{(3)}\to K(\mathbb{Z},2)$, so we can extend our bundle $E$ to $X^{(3)}$.
 For the extension of the bundle  to $X^{(n)}$, $n>3$, the reason is similar but more direct, since $K(\mathbb Z, 2)$ has trivial homotopy groups of dimension
 $\ge 3$.  
 \end{proof}

Now we are ready to prove Theorem \ref{main3.1'}. By some bundle reductions, we reach the place to apply Lemma \ref{2}, Lemma \ref{1}, Theorem \ref{basic}, and Proposition \ref{selfmap},   and to prove that there is a fiber-preserving map in the homotopy class of $\tilde f$.

\begin{proof}[Proof of Theorem \ref{main3.1'}]
Since $M_2$ is aspherical, by the homotopy exact sequence of fiber bundles, we have 
\begin{align}
\begin{CD}
1@> >>  \pi_1(S^1) @>  >>      \pi_1( E_2)@> (p_2)_* >> \pi_1(  M_2) @ > >> 1,
\end{CD}
\end{align}
and $E_2$ is also aspherical. 
 Below we use $\phi$ to denote the homomorphism $\tilde{f}_*\colon\pi_1(E_1)\to \pi_1(E_2)$.

\noindent{\bf Step I.} {\it Reduction to the case that $\pi_1(S^1)\to \pi_1(E_1)$ is injective and $\phi$ is injective on the fiber subgroup}.
If either of the desired conditions fails, since $\pi_1(S^1)\to \pi_1(E_2)$ is injective, $\tilde f(S^1)$ is null-homotopic in $E_2$. Then $\tilde{f}_*:\pi_1(E_1)\to \pi_1(E_2)$ factors through $(p_1)_*:\pi_1(E_1)\to \pi_1(M_1)$. Since $E_2$ is aspherical, up to homotopy, $\tilde{f}:E_1\to E_2$ factors through the bundle projection $p_1:E_1\to M_1$, and any map of the form 
$E_1\to M_1\to E_2$ is trivially fiber-preserving.

From now on, we assume that $\pi_1(S^1)\to \pi_1(E_1)$ is injective and $\phi$ sends the fiber subgroup of $\pi_1(E_1)$ to the fiber subgroup of $\pi_1(E_2)$ injectively.
Under this assumption, it is easy to see that we have the following commutative diagram: 
\begin{align}\label{3.1}
\begin{CD}
		1@> >> \pi_1(S^1) @>  >>	\pi_1( E_1) @>  (p_1)_*>> \pi_1(M_1)  @ > >> 1 \\
		@.	@VV\phi|  V  @VV \phi V   @VV \bar{\phi} V @.  \\
		1@> >>  \pi_1(S^1) @>  >>      \pi_1( E_2)@> (p_2)_* >> \pi_1(  M_2) @ > >> 1.
\end{CD}
\end{align}

\smallskip

\noindent{\bf Step II}. {\em Reduction to the aspherical case.}  
By Lemma \ref{2}, $E_1$ extends to a bundle $\widetilde E_1\rightarrow K(\pi_1(M_1),1)$ such that $\widetilde{E_1}$ is a $K(\pi_1(E_1),1)$-space and the map $E_1\rightarrow \widetilde{E_1}$ induces the identity homomorphism on $\pi_1$. Then we have the following commutative diagram

\begin{align}\label{3.2}
\begin{CD}
		1@> >> \pi_1(S^1) @>  >>	\pi_1( E_1) @>  (p_1)_*>> \pi_1(M_1)  @ > >> 1 \\
		@.	@VV{id}   V  @VV id=(i_{E_1})_* V   @VV id=(i_{M_1})_* V @.  \\
		1@> >>  \pi_1(S^1) @>  >>      \pi_1( \tilde E_1)@> (\tilde p_1)_* >> \pi_1(  X_1) @ > >> 1.
\end{CD}
\end{align}
Here $X_1=K(\pi_1(M_1), 1)$, $i_{M_1}: M_1\to X_1$ and $i_{E_1}: E_1\to \tilde E_1$ are the inclusion maps. Note that $i_{E_1}:E_1\to \tilde E_1$ is a bundle map which covers $i_{M_1}:M_1\to X_1.$

By the identification in (6), we can write (5) as
\begin{align}\label{3.3}
\begin{CD}
		1@> >> \pi_1(S^1) @>  >>	\pi_1( \tilde E_1) @>  (\tilde p_1)_*>> \pi_1(X_1)  @ > >> 1 \\
		@.	@VV\phi|  V  @VV \phi V   @VV \bar{\phi} V @.  \\
		1@> >>  \pi_1(S^1) @>  >>      \pi_1( E_2)@> (p_2)_* >> \pi_1(  M_2) @ > >> 1.
\end{CD}
\end{align}
Note that all spaces involved in diagram (\ref{3.3}) are aspherical. 

Since $E_2$ is aspherical, there exists $\tilde{f}_1:\tilde{E}_1\to E_2$ such that $\phi=(\tilde{f}_1)_*$ and $\tilde{f}_1$ is an extension of $\tilde{f}:E_1\to E_2$. If we know that $\tilde{f}_1:\tilde{E}_1\to E_2$ is homotopic to a fiber-preserving map, then we can restrict the homotopy to $E_1\subset \tilde{E}_1$ to show that $\tilde{f}:E_1\to E_2$ is also homotopic to a fiber-preserving map. So it suffices to work with the bundle $\tilde{E}_1\to X_1$ with aspherical base space $X_1$.

\noindent{\bf Step III}. {\em Reduction to bundles over $X_1$.}  
Since $M_2$ is aspherical, there exists a map $g\colon X_1\to M_2$, such that $g_*:\pi_1(X_1)\to \pi_1(M_2)$ equals $\bar{\phi}$. Let $p_1'\colon E_1'\to X_1$ be the pull-back bundle of $p_2\colon E_2\to M_2$ via $g\colon X_1\to M_2$. Then we have a fiber-preserving bundle map $\tilde{f}':E_1'\to E_2$, and the following diagram commutes.

\begin{align}\label{3.4}
\begin{CD}
	      1@> >> \pi_1(S^1) @>  >> \pi_1(  E_1') @>  (p_1')_*>> \pi_1(X_1) @ > >> 1   \\
		@.	@VVid  V @ VV \tilde{f}'_*  V   @VV g_*=\bar{\phi} V @. \\
		1@> >>  \pi_1(S^1) @>  >> \pi_1( E_2)@> (p_2)_* >> \pi_1(  M_2) @ > >> 1.
\end{CD}
\end{align}

Since the left vertical arrow in diagram (\ref{3.4}) is the identity, the right square of diagram (\ref{3.4}) gives the fiber product of $\bar{\phi}$ and $(p_2)_*$. Comparing the right squares of diagrams (\ref{3.3}) and (\ref{3.4}), by the universal property of the fiber product, there exists a homomorphism $\psi\colon\pi_1(\tilde E_1)\to \pi_1(E_1')$, such that $$\tilde{f_1}_*=\phi=\tilde{f}'_*\circ \psi\colon\pi_1(\tilde E_1)\to \pi_1(E_2), \ \ and  \ \ (\tilde p_1)_*=(p_1')_*\circ \psi\colon\pi_1(\tilde E_1)\to \pi_1(X_1).$$

Since $E_1'$ is  aspherical, there is a map $\tilde{h}\colon \tilde E_1\to E_1'$ such that $\tilde{h}_*=\psi\colon\pi_1(\tilde E_1)\to \pi_1(E_1')$, and we have 
$$\tilde{f_1}_*=\tilde{f}'_*\circ \psi=\tilde{f}'_*\circ \tilde{h}_*=(\tilde{f}'\circ \tilde{h})_*.$$ 
Since $E_2$ is aspherical, $\tilde{f_1}$ and $\tilde{f}'\circ \tilde{h}$ are homotopic to each other. Since $\tilde{f}'$ is fiber-preserving, it suffices to prove that $\tilde{h}\colon \tilde E_1\to E_1'$ is homotopic to a fiber-preserving map.

Now we have the following commutative diagram:
\begin{align}\label{3.5}
\begin{CD}
	      1@> >> \pi_1(S^1) @>  >> \pi_1( \tilde E_1) @>  (\tilde p_1)_*>> \pi_1(X_1) @ > >> 1   \\
		@.	@VV\psi|  V @ VV \tilde{h}_*=\psi V   @VV id V @. \\
		1@> >>  \pi_1(S^1) @>  >> \pi_1( E_1')@> (p_1')_* >> \pi_1(  X_1) @ > >> 1.
\end{CD}
\end{align}
Here $\psi|: \pi_1(S^1)\to \pi_1(S^1)$ is a homomorphism that sends $1$ to $k\in \mathbb{Z}\cong \pi_1(S^1)$. Actually, we have $k\ne 0$ since $(\tilde{f}_1)_*$ is injective on the fiber subgroup and $(\tilde{f}_1)_*=\tilde{f}'_*\circ\psi$. Up to changing orientation, we can assume that $k> 0$.

\smallskip

\noindent{\bf Step IV}. {\em Reduction to the case $k=1$.}
By the commutative diagram (\ref{3.5}), we can check that $\tilde{h}_*=\psi\colon\pi_1(\tilde E_1)\to \pi_1(E_1')$ is injective,  $\pi_1(S^1)\cap \psi(\pi_1(\tilde E_1))=k\mathbb{Z}<\pi_1(S^1)\cong \mathbb{Z}$, and $[\pi_1(E_1')\colon\psi(\pi_1(\tilde E_1))]=k$. Let $q\colon E_1''\to E_1'$ be the $k$-sheet covering of $E_1'$ corresponding to $\psi(\pi_1(\tilde E_1))<\pi_1(E_1')$. Then $E_1''$ has an induced $S^1$-bundle structure over $X_1$, and $q$ is a fiber-preserving map. 

Let $\tilde{k}\colon \tilde E_1\to E_1''$ be the lifting map such that $\tilde{h}=q\circ \tilde{k}\colon \tilde E_1\to E_1'$. It suffices to prove that $\tilde{k}$ is homotopic to a fiber-preserving map. On the group level, by diagram (\ref{3.5}), we have the following commutative diagram where  $\tilde{k}_*\colon\pi_1(\tilde E_1)\to \pi_1(E_1'')$ is an isomorphism. 
\begin{align}\label{3.6}
\begin{CD}
	      1@> >> \pi_1(S^1) @>  >> \pi_1(\tilde  E_1) @>  (\tilde p_1)_*>> \pi_1(X_1) @ > >> 1   \\
		@.	@VVid  V @ VV \tilde{k}_* V   @VV id V @. \\
		1@> >>  \pi_1(S^1) @>  >> \pi_1( E_1'')@> (p_1'')_* >> \pi_1(  X_1) @ > >> 1,
\end{CD}
\end{align}

\smallskip

\noindent{\bf Step V}. {\em Finishing the proof.}
Since $X_1$ is aspherical, we have a natural isomorphism $H^2(\pi_1(X_1);\mathbb{Z})\cong H^2(X_1;  \mathbb{Z})$. Let $e_1,e_1''\in H^2(X_1;\mathbb{Z})\cong H^2(\pi_1(X_1); \mathbb{Z})$ be the Euler classes of the oriented $S^1$-bundles $\tilde p_1\colon \tilde E_1\to X_1$ and $p_1''\colon E_1''\to X_1$, respectively.  The commutative diagram (\ref{3.6}) implies $e_1=e_1''$ by Lemma \ref{1}, thus $\tilde{E}_1$ is isomorphic to $E_1''$. Then Theorem \ref{basic} (ii) implies that $\tilde p_1\colon \tilde E_1\to X_1$ and $p_1''\colon E_1''\to X_1$ are isomorphic oriented $S^1$-bundles over $X_1$. 

Finally Proposition \ref{selfmap} implies the existence of a (fiber-preserving) bundle isomorphism $\tilde{k}'\colon \tilde E_1\to E_1''$, such that $\tilde{k}'_*=\tilde{k}_*\colon\pi_1(\tilde E_1)\to \pi_1(E_1'')$. Since $E_1''$ is aspherical, $\tilde{k}$ is homotopic to the fiber-preserving map $\tilde{k}'$. This finishes the proof.
\end{proof}

\section{Mapping degree sets of fiber-preserving maps between $S^1$-bundles}

In this section, we will prove Theorem \ref{main1}. 

\smallskip

First we state a fact about vertical maps that will be used below.

\begin{lem}\label{h-v} 
Let  $M_i$ be a closed oriented $n$-manifold, and let 
$S^1\overset{j_i}\to E_i\ \overset{p_i}\to M_i$ be an oriented $S^1$-bundle, $i=1,2$. 
Suppose  $ f\colon E_1\to E_2$ is a fiber-preserving non-zero degree map that induces $\bar f\colon M_1\to M_2$. Then we have a factorisation
  $$ f\colon E_1\overset{v}\to  E_1'\overset{f'}\to E_2,$$
where $v$   is a vertical map, and $ f'$ is a bundle map that induces $\bar f\colon M_1\to M_2$.
\end{lem}

\begin{proof}
This can be proved by a standard pull-back argument. Let $E_1'$ be the pull-back bundle of $E_2\xrightarrow{p_2}M_2$ via $\bar{f}\colon M_1\to M_2$, where 
$$E_1'=\{(x_1,e_2)\ |\ x_1\in M_1, e_2\in E_2, \bar{f}(x_1)=p_2(e_2)\}.$$ Then we have an $S^1$-bundle $q\colon E_1'\to M_1$ over $M_1$ defined by $q(x_1,e_2)=x_1$, and a bundle map $f'\colon E_1'\to E_2$ defined by $f'(x_1,e_2)=e_2$ that induces $\bar{f}\colon M_1\to M_2$.

We define $v\colon E_1\to E_1'$ by $v(e_1)=(p_1(e_1),f(e_1))$. This is a fiber-preserving map that induces the identity map on $M_1$ and satisfies $f=f'\circ v$.
\end{proof}

From now on,  we will often use $\tilde M_a$ to denote the oriented circle bundle over a closed oriented $n$-manifold $M$ with Euler class $e(\tilde M_a)=a\in H^2(M;\mathbb{Z})$. 

\begin{prop}\label{cal1}
Suppose  $M$ is a closed oriented $n$-manifold,  $a, b \in H^2(M;\mathbb{Z})$ and $k\ne 0$ is an integer. Then the following are equivalent.
\begin{enumerate}
\item $ka=b$.
\item There exists a vertical  map   $\tilde M_{a}\to \tilde M_{b}$ of degree $k$.
\end{enumerate}
 \end{prop}
 
\begin{proof} 
We first recall the obstruction definition of the Euler class of an oriented $S^1$-bundle $p_a\colon\tilde{M}_a\to M$. 	

First we assume that $M$ admits a CW-complex structure $X$, denote its $1$-skeleton by $X^{(1)}$, and take a section $s\colon X^{(1)}\to \tilde{M}_a$ such that $p_a\circ s$ equals the inclusion $X^{(1)}\to M$. The Euler class $a\in H^2(M;\mathbb{Z})$ is represented by a cellular $2$-cocycle $\alpha\in C^2(X)$ defined as follows. For any $2$-cell $\Delta^2$ of $X$ with an orientation, $\alpha(\Delta^2)$ equals the number of $s(\partial \Delta^2)$ wrapping around the $S^1$-fiber. This makes sense since the restriction of the $S^1$-bundle on $\Delta^2$ is trivial.

Now we prove the proposition.

\noindent $(2) \Rightarrow (1)$: If we have a vertical map $v\colon \tilde M_a\to \tilde M_b$ of degree $k$, then the restriction of $v$ on each $S^1$-fiber of $\tilde{M}_a$ is a degree-$k$ map to the image $S^1$-fiber of $\tilde{M}_b$. We take a section $s\colon X^{(1)}\to \tilde{M}_a$ of $\tilde{M}_a\to M$ as in the obstruction definition of Euler class. Since $v\colon\tilde{M}_a\to \tilde{M}_b$ is vertical, $v\circ s\colon X^{(1)}\to \tilde{M}_b$ is a section 
of $\tilde{M}_b\to M$. For any $2$-cell $\Delta^2$ of $X$ with an orientation, the number of $v\circ s(\partial \Delta^2)$ wrapping around the $S^1$-fiber of $\tilde{M}_b$ equals $k$ times the number of $s(\partial \Delta^2)$ wrapping around the $S^1$-fiber of $\tilde{M}_a$. So we have $\beta=k\alpha \in C^2(X)$, and $b=ka\in H^2(M;\mathbb{Z})$ holds.

\noindent $(1) \Rightarrow (2)$: Suppose that $b=ka \in H^2(M;\mathbb{Z})$. We first suppose that $k> 0$ holds. By Theorem \ref{basic} (iii),  we can consider all $S^1$-bundles as principal  $\text{SO}(2)(\cong S^1)$-bundles.
Then we have a group embedding
$$\mathbb{Z}_k\subset S^1 \subset \text{Homeo}_0(\tilde M_a),$$ and the $\mathbb{Z}_k$-action induces an $S^1$-bundle 
$$S^1/\mathbb{Z}_k \to \tilde M_a/\mathbb{Z}_k \cong \tilde M_c \to M$$ for some $c\in H^2(M;\mathbb{Z})$,
and we have a vertical map $v_k\colon \tilde M_a \to \tilde M_c$ of degree $k$. By the previous paragraph, we have $c=ka$, thus $c=b$. So $v_k$ is a vertical map from $\tilde{M}_a$ to $\tilde{M}_b$ of degree $k$, as desired. 

If $k<0$, we flip the fiber orientation of $\tilde{M}_a$, and denote the resulting bundle by $\tilde{M}_{-a}$. This notation makes sense since its Euler class is $-a\in H^2(M;\mathbb{Z})$. Moreover, $\tilde{M}_{-a}$ is simply the orientation reversal of $\tilde{M}_a$. A degree-$k$ vertical map $\tilde{M}_a\to \tilde{M}_{b}$ simply gives a degree-$(-k)$ vertical map $\tilde{M}_{-a}\to \tilde{M}_{b}$. Since $-k\geq 0$, the previous paragraph gives $b=(-k)(-a)=ka$. 

For a general $M$, we can take a homotopy equivalence $h\colon X\to M$ from a CW-complex $X$ to $M$ (see for example Corollary A.12 of \cite{Ha}), and apply a similar argument as above.
\end{proof}

We rewrite Theorem \ref{main1} as the statement below.

\begin{thm}[Theorem \ref{main1}]\label{main2.1} Suppose $M$ and $N$ are closed oriented $n$-manifolds,
 $a\in H^2(M;\mathbb{Z})$ and $b\in H^2(N;\mathbb{Z})$.
Then the mapping degree set of fiber-preserving maps from $\tilde M_a$ to $\tilde N_b$ is given by
$$D_{FP}(\tilde M_a, \tilde N_b)=\{0\} \cup\{k\cdot \deg(f)\ |\ k\ne 0,  f\colon M\to N, \deg(f)\ne 0 \, \text{such that}\, f^\#(b)=ka\},$$
where $f^\#\colon H^2(N;\mathbb{Z})\to H^2(M;\mathbb{Z})$ is  the induced homomorphism.
\end{thm}

\begin{proof} 
We first prove that the left-hand set is a subset of the right-hand set.
Suppose  $\tilde f\colon \tilde M_a\to \tilde N_b$ is a fiber-preserving map which induces $f\colon M\to N$.
Then by Lemma \ref{h-v}, we have $\tilde f=\tilde f_2\circ \tilde f_1$, where $\tilde f_1 \colon \tilde M_a\to \tilde M_c$  is a vertical map, and $\tilde f_2\colon \tilde M_c \to \tilde N_b$ is a bundle map that induces $f\colon M\to N$. For the bundle map, we have 
$e(\tilde M_c)=f^\#(e(\tilde N_b)),$ that is $c=f^\#(b)$.
For the vertical map, we have $ka=c$ by Proposition \ref{cal1}, where $k$ is the degree of $\tilde f_1$. So we have
$$ka=f^\#(b).$$

By the product formula of mapping degrees of fiber bundles, we have 
$$\deg (\tilde f)=k\cdot \deg (f).$$ If either $k$ or $\text{deg}(f)$ is $0$, we have $\text{deg}(\tilde{f})=0$. Otherwise, $\text{deg}(\tilde{f})$ lies in the second set on the right-hand side.
The proof of the first part is done.

Next, we prove the converse inclusion. At first, $0$ clearly belongs to $D_{FP}(\tilde{M}_a,\tilde{N}_b)$, since any constant map $\tilde{M}_a\to \tilde{N}_b$ is fiber-preserving and has degree $0$. Now we work on any non-zero number on the right-hand side.
Suppose $f\colon M\to N$ is a non-zero degree map and $f^\#(b)=ka$. Then we take the pull-back bundle of $\tilde{N}_b\to N$ as in diagram (\ref{d1}).
Since the map $\tilde f$ in (\ref{d1}) is a bundle map, we have 
$$e( f^*(\tilde N_b))=f^\#(e(\tilde N_b))= f^\#(b),$$
where $f^\#\colon H^2(N;\mathbb{Z})\to  H^2(M;\mathbb{Z})$ is the induced homomorphism. It implies that 
$$ f^*(\tilde N_b)= \tilde M_{f^\#(b)} = \tilde M_{ka}. $$
So we have the right square of the following diagram (\ref{4}). 

\begin{align}\label{4}
\begin{CD}
		 \tilde M_{a} @>  \tilde v>>	 \tilde M_{ka} @>  \tilde f>> \tilde N_b  \\
			@VV p_0 V @VV p_1 V   @VV p_2 V  \\
		M @> \text{id} >>       M @> f >> N
\end{CD}
\end{align}

By Proposition \ref{cal1}, we have a vertical map  $\tilde v\colon\tilde M_{a} \to   \tilde M_{ka}$ of degree $k$, so we have the left square of (\ref{4}). 
Note also that $\deg(\tilde f)=\deg{f}$.
Then  
$$\deg(\tilde f \circ \tilde v)=\deg(\tilde v)\cdot \deg(\tilde f)=k\cdot \deg(f),$$
and the proof of the second inclusion is completed. \end{proof}

We end our discussion in this section with the next two observations, which will make the picture of fiber-preserving maps between $S^1$-bundles more complete.

 \begin{lem}\label{not torsion} Suppose $M$ and $N$ are closed oriented $n$-manifolds and $a\in H^2(M;\mathbb{Z})$, $b\in H^2(N;\mathbb{Z})$. 
If there is a fiber-preserving  map $\tilde f\colon \tilde M_a\to \tilde N_b$  of non-zero degree, then $a$ is torsion if and only if
$b$ is torsion.
\end{lem}

\begin{proof}
Let $k$ be the degree of $\tilde{f}$ on the $S^1$-fiber, and let $f$ be the induced map on the base manifolds $M\to N$. Then $ka= f^\#(b)$
by the proof of Theorem \ref{main1}. Since $\tilde f$ is a non-zero degree map, $f$ is also a non-zero degree map and $k\ne 0$. 

If $b$ is a torsion class, then $ f^\#(b)$ is torsion. So $a$ is torsion since $k\ne 0$, and we proved one direction of the lemma. Conversely, if $a$ is torsion, then $ f^\#(b)$ is torsion. Since $ f$ is a non-zero degree map, the induced homomorphism on cohomology modulo torsion $$f^\#\colon H^*(N;\mathbb{Z})/tors \to H^*(M;\mathbb{Z})/tors$$ is injective. So $b$ is torsion, and we proved the other direction of the lemma.
\end{proof}

We define the {\em mapping degree set of vertical  maps} as
\[
D_{V}(\tilde M_a,\tilde M_b):=\{d\in\Z \ |\, \text{there exists a vertical map} \ \tilde f\colon \tilde M_a\to \tilde M_b, \ \deg(\tilde f)=d\}.
\]

Clearly $D_{V}(\tilde M_a,\tilde M_b)\subset  D_{FP}(\tilde M_a,\tilde M_b).$

\begin{cor}\label{cal2}
Suppose  $a, b \in H^2(M;\mathbb{Z})$. Then $D_{V}(\tilde M_a,\tilde M_b)$ is 
\begin{itemize}
\item[(1)] the empty set,  if $b\notin \left< a \right>$;  
\item[(2)] an infinite set,
if $b=ka$, $k\ne 0$ and $a$ is torsion;
\item[(3)] $\{ k\}$, if $b=ka$, $k\ne 0$ and $a$ is not torsion.
\end{itemize}
 \end{cor}
 
 \begin{proof} \
 \begin{itemize}
\item[(1)] This clearly follows from Proposition \ref{cal1}.
 \item[(2)] Assume $l>0$ is the minimal integer such that $la=0$. Then $b=(k+ql)a$ for any integer $k+ql\ne 0$.
 By Proposition \ref{cal1}, there exists a  vertical  map   $\tilde M_{a}\to \tilde M_{b}$ of degree $k+ql$.
 \item[(3)]   By Proposition \ref{cal1}, there exists  a vertical  map   $\tilde M_{a}\to \tilde M_{b}$ of degree $k$.
 Assume there exists  a vertical  map   $\tilde M_{a}\to \tilde M_{b}$ of degree $k'$.
 Then by Proposition \ref{cal1}, $b=k'a$, which implies that $ka=k'a$. Since $a$ is not torsion, we have $k=k'$.
 \end{itemize}
 \end{proof}

\section{Applications}
The following consequence of Theorem \ref{main2} is essential in our further applications.
\begin{cor}\label{fiber3}
Let 
$S^1\overset{j_i}\to E_i\ \overset{p_i}\to M_i$ be an oriented $S^1$-bundle over a closed oriented $n$-manifold $M_i$ with $i=1,2$, and assume that $M_2$ is aspherical. Let $\tilde f\colon   E_1\to  E_2$ be a map of non-zero degree.
If each finite-index subgroup of  $\pi_1(M_2)$ is center-free,
 then $\tilde f$ is homotopic to a fiber-preserving map. As a consequence, we have 
$$D_{FP}(E_1 ,E_2)=  D(E_1, E_2).$$
\end{cor}

\begin{proof} By Theorem \ref{main2}, we need only to verify that 
 the induced map $\tilde f_*$ on $\pi_1$ sends the fiber group to the fiber group injectively.

Since $\tilde{f}\colon E_1\to E_2$ is a non-zero degree map, $\tilde{f}_*(\pi_1(E_1))<\pi_1(E_2)$ is a finite-index subgroup. Let $q\colon E_2'\to E_2$ be the finite cover of $E_2$ corresponding to $\tilde{f}_*(\pi_1(E_1))$. 
Then $E_2'$ has an induced oriented $S^1$-bundle structure $p_2'\colon E_2'\to M_2'$, where the base manifold $M_2'$ is a finite cover of $M_2$, and $q$ is a fiber-preserving map. Note that both  $M_2'$ and $E_2'$ are still aspherical.
Now we have a lifting map $\tilde{g}\colon E_1\to E_2'$, such that $\tilde{f}=q\circ \tilde{g}$, and  $\tilde{g}_*\colon \pi_1(E_1)\to \pi_1(E_2')$ is surjective.

The fiber group $\pi_1(S^1)$ must be mapped  into $\pi_1(E_1)$ injectively.
Otherwise, since $E'_2$ is aspherical, the map $\tilde g: E_1\to E_2'$ factors though $M_1$ (see the Step I of the proof of Theorem \ref{main3.1'}), which contradicts that 
$\tilde g$ is of non-zero degree.  We still denote the image of fiber group $\pi_1(S^1)\subset \pi_1(E_1)$ by $\pi_1(S^1)$, which lies in the center of $\pi_1(E_1)$.

As a finite-index subgroup of  $\pi_1(M_2)$, $\pi_1(M_2')$ is center-free,  $(p_2')_*\circ \tilde g_* (\pi_1(S^1))<\pi_1(M_2')$ must be the trivial subgroup, 
and $\tilde g_*(\pi_1(S^1))<\pi_1(E_2')$ must be contained in the fiber subgroup $\pi_1(S^1)<\pi_1(E_2')$. 
So  $\tilde f_*(\pi_1(S^1))<\pi_1(E_2)$ must be contained in the fiber subgroup $\pi_1(S^1)<\pi_1(E_2)$. 
\end{proof}

The next result is well-known in the theory of hyperbolic manifolds and 3-manifold topology.

\begin{lem}\label{SCF}
Suppose $N$ is either a closed hyperbolic $n$-manifold, or a closed irreducible 3-manifold that is not a Seifert manifold.
Then each finite-index subgroup of  $\pi_1(N)$ is center-free.
\end{lem}

\begin{proof} For a manifold $N$ as in Lemma \ref{SCF}, any finite-index subgroup $H<\pi_1(N)$ is the fundamental group of a finite cover $M\to N$.

If $N$ is a closed hyperbolic $n$-manifold, $M$ is also a closed hyperbolic $n$-manifold. Then Corollary 1 in \cite[Section 12.2]{Ra} implies that $H=\pi_1(M)$ is center-free. 

If $N$ is a closed irreducible 3-manifold that is not a Seifert manifold,
then $N$ is either a hyperbolic manifold, or a Sol manifold, or a 3-manifold with nontrivial JSJ decomposition by Thurston's picture of 3-manifolds (verified by Perelman and Thurston).
 Then $M$ is a 3-manifold of the same type. According to  \cite[Theorem 2.5.5]{AFW},  $H=\pi_1(M)$ is center-free.
\end{proof}

\subsection{Finiteness of mapping degree sets between $S^1$-bundles}

Now we prove Theorem \ref{appl1}.

\begin{proof}[Proof of Theorem \ref{appl1}] To prove that  $D(\tilde M_a,\tilde N_b)$ is a finite set, we need only to prove that there are only finitely many non-zero elements  in $D(\tilde M_a,\tilde N_b)$. So below we only consider the non-zero mapping degrees.

Let $\tilde f\colon \tilde M_a\to \tilde N_b$ be a map of non-zero degree. We will check that there are only finitely many possibilities for $\deg(\tilde f)$.
Since $N$ is a closed hyperbolic manifold, it is aspherical and each finite-index subgroup of  $\pi_1(N)$ is center-free by Lemma \ref{SCF}. By Corollary \ref{fiber3}, we may assume that $\tilde f$ is a fiber-preserving map and induces $f\colon M\to N$. By Theorem \ref{main1}, we have $\deg(\tilde f)=k \deg(f)$ and $ka=f^\#(b)$.

 Since $b$ is not torsion and $\tilde f\colon \tilde M_a\to \tilde N_b$ is a map of non-zero degree, $a$ is not torsion by Lemma \ref{not torsion}, so there is at most one integer $k$ such that $ka=f^\#(b)$. To prove  that there are only finitely many possibilities for $\deg(\tilde f)$, we only need to prove that there are only finitely many possibilities for both $\deg (f)$ and $f^\#(b)$.

Since $N$ is a closed oriented hyperbolic $n$-manifold, its simplicial volume satisfies $\|N\| >0$  and  $|\deg(f)|\cdot \|N\|\le\|M\|$ for any map $f\colon M\to N$ \cite[6.1.4, 6.1.2]{Th}, thus $$|\deg(f)|\le \|M\|/\|N\|.$$ 
So $\deg(f)$ can only take finitely many values. Below we prove that there are only finitely many possibilities for $\tilde \beta=f^{\#}(b)$.

Since $H_2(M;\mathbb{Z})$ is finitely generated, we choose 
a finite generating set $\{\alpha_1,..., \alpha_m\}$  of $H_2(M;\mathbb{Z})$.
By the Universal Coefficient Theorem, each cohomology class $\tilde \beta\in H^2(M;\mathbb{Z})$ is determined by 
$\{\tilde \beta (\alpha_1),... , \tilde \beta (\alpha_m)\} \in \mathbb{Z}^m$, up to a finite ambiguity with size $|\text{Tor}(H^2(M;\mathbb{Z}))|.$

\smallskip

\noindent {\bf Claim:} For each $1 \le i \le m$, there are only finitely many possibilities for $f^\#(b) (\alpha_i)$.

\noindent {\em Proof of Claim}:  We have 
$$f^\#(b)(\alpha_i)=b(f_\#(\alpha_i)).$$
So, we only need to prove that there are only finitely many possibilities for $f_\#(\alpha_i)$.
Let 
$$L=\max\{\|\alpha_i\| \ | \ i=1, ... , m\}.$$
By the functorial property of the simplicial volume (cf.~\cite[p.8]{Gr}), we have
$$L\ge \|\alpha_i\|\geq \|f_\#(\alpha_i)\|.$$
Since $N$ is a closed orientable hyperbolic manifold, the simplicial volume is a genuine norm on the finite-dimensional space $H_2(N,\mathbb{R})$, see \cite[Theorem 1.6]{CS}  for example. Hence there are only finitely many integer homology classes whose image in $H_2(N,\mathbb{R})$ has simplicial volume at most $L$.
This proves the Claim, thus also proves the theorem.
\end{proof}

\subsection{Realizing finite sets of integers as mapping degree sets}
 
 In this section, we will prove the Realization Theorem \ref{realization}. 
 First, we prove the following: 
 
 \begin{prop}\label{appl2}
 Suppose $N$ is a closed oriented aspherical $n$-manifold, every finite-index subgroup of  $\pi_1(N)$ is center-free,  and $b\in H^2(N; \mathbb{Z})$ is not torsion.
\begin{itemize} 
\item[(1)] Assume that   $D(N)$ is finite.   Then for any non-zero integers $k$ and $m$, we have
 $$
D(\tilde N_{mb},\tilde N_{kb})\subset \left\{\begin{array}{ll}
       \{0, \pm k/m\}, & \text{if $k$ is a multiple of $m$,} \\
       \text{} & \\
        \{0\}, & \text{otherwise}.
        \end{array}\right.
$$ 
\item[(2)] Assume that $D(N)=\{0,1\}$ and $f^{\#}(b)=b$ holds
for any degree one self-map $f\colon N\to N$.
 Then for any non-zero integers $k$ and $m$, we have
 $$
D(\tilde N_{mb},\tilde N_{kb})= \left\{\begin{array}{ll}
       \{0,k/m\}, & \text{if $k$ is a multiple of $m$,} \\
       \text{} & \\
        \{0\}, & \text{otherwise}.
        \end{array}\right.
$$ 
  \end{itemize}
  Moreover, the conclusion in (2) holds if $N$ is also a closed oriented hyperbolic $n$-manifold with $n\ge 3$ and the order of its isometry group $\mathrm{Isom}(N)$ is odd. 
 \end{prop}

\begin{proof}
Suppose $\tilde f\colon \tilde N_{mb}\to \tilde N_{kb}$ is a map of non-zero degree for some non-zero integer $k$. Since $N$ is aspherical and every finite-index subgroup of  $\pi_1(N)$ is center-free,
by Corollary \ref{fiber3}, we may assume that $\tilde f$ is a fiber-preserving map that induces $f\colon N\to N$.
By Theorem \ref{main1}, we have 
\begin{align}\label{9}
\deg(\tilde f)=l \deg(f) \  \text{and} \ l(mb)=f^\#(kb)
\end{align}
for some non-zero integer $l$.
Since $\tilde f$ is of non-zero degree, $f$ is also of non-zero degree. 

\smallskip

\noindent (1) Since $D(N)$ is finite, we have $\deg(f)=\pm 1$.   By (\ref{9}), we have $\deg(\tilde f)=\pm l$, and  
\begin{align}\label{10}
lmb=f^\#(kb)=kf^\#(b).
\end{align}  
We use $b_{\mathbb{Q}}$ to denote the image of $b\in H^2(N;\mathbb{Z})$ in $H^2(N;\mathbb{Q})$, and $b_{\mathbb{Q}}\ne 0$ holds since $b$ is not torsion. Then we obtain 
\begin{align}\label{11}
f^\#(b_{\mathbb{Q}})=\frac {ml}k b_{\mathbb{Q}}=\lambda b_{\mathbb{Q}} \in H^2(N; \mathbb{Q}).
\end{align}
We will show that the rational number $\lambda$ is $\pm 1$. Since $f\colon N\to N$ is a degree $\pm1$ map, it induces surjections $f_\#\colon H_i(N;\mathbb{Z})\to H_i(N;\mathbb{Z})$ for all $i\ge 0$.
Since each self-surjection on each finitely generated Abelian group is an isomorphism, 
 $f$ induces isomorphisms $f_\#\colon H_i(N;\mathbb{Z})\to H_i(N;\mathbb{Z})$ for all $i\ge 0$. By algebraic duality, $f$ induces isomorphisms $f^\#\colon H^i(N;\mathbb{Z})\to H^i(N;\mathbb{Z})$ for all $i\ge 0$, and  in particular $f^\#\colon H^2(N;\mathbb{Z})\to H^2(N;\mathbb{Z})$ is an isomorphism. Note that (\ref{11}) implies that $b$ is an eigenvector of $f^\#\colon H^2(N;\mathbb{Z})/tors\to H^2(N;\mathbb{Z})/tors$ with rational eigenvalue $\lambda$. 
Since $f^\#$ is an isomorphism, the characteristic polynomial of $f^\#$ is an integer polynomial with leading coefficient 1 and constant  $\pm 1$, which is also the characteristic polynomial for $f^\#_{\mathbb{Q}}: H^2(N;\mathbb{Q})\to H^2(N;\mathbb{Q})$. 
Then  this rational eigenvalue $\lambda$ has to be $\pm 1$, thus $k=\pm ml$ holds. So $k$ is a multiple of $m$, and  $\deg(\tilde f)\in \{\pm k/m\}$. Then we have
$$D(\tilde N_{mb},\tilde N_{kb})\subset \{0, \pm k/m\}.$$
In particular,
$$D(\tilde N_b)=D(\tilde N_b,\tilde N_{b})\subset \{0, \pm 1\}.$$

\smallskip

\noindent (2) Now we have that $\deg (f)=1$ and  $f^{\#}(b)=b$. 
 Assume that $D(\tilde N_{mb},\tilde N_{kb})$ contains a non-zero integer $l$. By (\ref{9}), applying the same argument  to the present situation, we have $$l\cdot mb=f^{\#}(kb)=kf^{\#}(b)=kb,$$
Since $b$ is not torsion,  we have  $k=ml$, that is $k$ is an integer multiple of $m$, and $\deg(\tilde f)=l=k/m$. Then Proposition \ref{cal1} implies that $k/m\in D(\tilde N_{mb},\tilde N_{kb})$ if $k$ is a factor of $m$.
So $D(\tilde N_{mb},\tilde N_{kb})$ is 
       $\{0,k/m\}$ if $k$ is a multiple of $m$, and is  
       $\{0\}$ otherwise.

Now we prove the ``Moreover" part of this proposition.
Since $N$ is a closed oriented hyperbolic $n$-manifold, by Corollary \ref{fiber3}, $\tilde f:\tilde{N}_{mb} \to \tilde{N}_{kb}$ induces a map $f\colon N\to N$ with $|\deg(f)|=1$.
By  \cite[Theorem 6.4]{Th},  every map $f\colon N\to N$ of $|\deg(f)|=1$ is homotopic to an isometry when $n\geq3$.
Since the order of $\text{Isom} (N)$ is odd, $f$ is homotopic to an isometry of odd order, and so  $\deg(f)=1$. 
We may assume that $f^r=id$ for some odd $r$.

Again by  (\ref{11}), $ml/k$ is a real eigenvalue of $f^\#: H^2(N, \mathbb Q)\to H^2(N,\mathbb Q)$. 
Since $(f^\#)^r =id$, $ml/k$ is an $r$th root of unity. 
Since $r$ is odd, we have $ml/k=1$, that is $k=ml$. So
$f^\#(b)=b$ holds by (\ref{10}).
\end{proof}

Next, we are going to prove Theorem \ref{0,k}. By Proposition \ref{appl2} we only need to prove the following result.

\begin{prop}\label{0,k1}
 For $n=4, 5$,  there exists a closed, orientable $(n-1)$-manifold $N$ such that the following hold.

(a) $N$ is aspherical and every finite-index subgroup of  $\pi_1(N)$ is center-free.

(b) $D(N)=\{0,1\}$, 

(c) If $n=4$, there is a non-torsion class $b\in H^2(N;\mathbb{Z})$, such that $f^{\#}(b)=b$ for any degree-one map $f:N\to N$.
If $n=5$, $N$ is a hyperbolic $(n-1)$-manifold such that $\text{Isom}(N)$ has odd order and $\beta_2(N)>0$.
 \end{prop}

\begin{proof}
We split the proof into the cases $n=4$ and $n=5$.

\smallskip

\noindent (1) {\em The case $n=4$}.  Let $K_1$ and $K_2$ be two hyperbolic knots in the 3-sphere $S^3$, such that $\mathrm{Isom}(E(K_1))=\{id\}$  and $E(K_1)$ and $E(K_2)$ are not homeomorphic to each other.
Here $E(K_i)$ denotes the knot complement $S^3\setminus N(K_i)$. Let $S_1$ and $S_2$ be Seifert surfaces of $K_1$ and $K_2$, respectively. Here we choose $K_1$ so that all self-homeomorphisms of $E(K_1)$ are homotopic to the identity. For example, the $9_{32}$ knot can serve the role of $K_1$, as proved in \cite{Ri}.
Let $$M=E(K_1)\cup_\phi E(K_2)$$ be the closed oriented 3-manifold obtained by taking an orientation reversing 
homeomorphism $\phi\colon\partial E(K_1)\to \partial E(K_2)$ such that $\phi (\partial S_1)=\partial S_2$. 
Clearly, $M$ is irreducible and is not a Seifert manifold, so $M$ meets condition (a) by Lemma \ref{SCF}. 
By classical results in 3-manifold topology, the following statements hold.
\begin{itemize}
\item[(i)] $H_2(M;\mathbb{Z})=\mathbb{Z}$, and indeed it is generated by $S=S_1\cup S_2$.
\item[(ii)] $M=E(K_1)\cup_{\phi}E(K_2)$ gives the JSJ decomposition of $M$, where the JSJ torus $T$ is the image of $\partial E(K_i)$. Note that $T$ is separating in $M$, and $M$ is Haken. In particular, $M$ is aspherical and every finite-index subgroup of  $\pi_1(M)$ is center-free by Lemma \ref{SCF}.
\item[(iii)] The simplicial volume of $M$ satisfies $\|M\|= \|E(K_1)\|+\|E(K_2)\|>0$.
\end{itemize}

Suppose $f\colon M\to M$ is a map of non-zero degree. Since $\|M\|>0$, we have $\deg(f)=\pm 1$, and $f_*:\pi_1(M)\to \pi_1(M)$ is surjective. Since $3$-manifold groups are Hopfian, $f_*$ must be an isomorphism. Since $M$ is aspherical, $f$ is actually a homotopy equivalence, .
Moreover, since $M$ is Haken, $f$ is homotopic to a homeomorphism, and we assume that $f$ is a homeomorphism in the following. 

By the JSJ theory, we may assume that 
the JSJ decomposition $M=E(K_1)\cup_{\phi}E(K_2)$ is invariant under $f$. Since $E(K_1)$ and $E(K_2)$ are not homeomorphic to each other, each $E(K_i)$ is invariant under $f$ for $i=1,2$.
Since the interior of $E(K_1)$ is a cusped hyperbolic $3$-manifold, we may assume that $f|_{E(K_1)}$ is a self-isometry, 
and we have that $f|_{E(K_1)}$ is the identity since $\mathrm{Isom}(E(K_1))=\{id\}$. 

So we have $\deg (f)=1$, and $f$ maps the oriented meridian of $E(K_1)$ to itself. Since $H_1(M;\mathbb{Z})\cong \mathbb{Z}$ is generated by the oriented meridian of $E(K_1)$, $f_{\#}\colon H_1(M;\mathbb{Z})\to H_1(M;\mathbb{Z})$ is the identity. Then the Poincar\'e duality and the fact that $\text{deg}(f)=1$ imply that $f^\#\colon H^2(M;\mathbb{Z})\to H^2(M;\mathbb{Z})$ is the identity.
Therefore $f^\#(b)=b$ for any $b\in H^2(M;\mathbb{Z})\cong \mathbb{Z}$.

\smallskip
 
\noindent (2) {\em The case $n=5$}. By a result of Belolipetsky and Lubotzky \cite[Theorem 1.1]{BL}, for each $k\geq2$, and any finite group $\Gamma$, there exists a closed hyperbolic $k$-manifold $M$ such that $\mathrm{Isom}(M)=\Gamma$. Indeed $M$ is an orientable manifold, which is observed by Weinberger  (see \cite[Section 3]{Mu}).

Let $M_i$ be a closed oriented hyperbolic $k$-manifold such that $\mathrm{Isom}(M_i)\cong\Z_{2i+1}$, the cyclic group of order 
$2i+1$. Then the family $\{M_i\}$ contains infinitely many hyperbolic $k$-manifolds.
If $k>3$, it follows from H. C. Wang's theorem \cite{Wa} that $\{\text{Vol}(M_i)\}_{i=1}^{\infty}$ is unbounded.
Here we take $k=n-1=4$. By the Gauss-Bonnet Theorem, the set of Euler characteristics $\{\chi(M_i)\}_{i=1}^{\infty}$, is unbounded from above.
Hence the set of second Betti numbers $\{\beta_2(M_i)\}_{i=1}^{\infty}$ is unbounded.
So there exists a hyperbolic 4-manifold $M$ such that the order of $\mathrm{Isom}(M)$ is odd and $\beta_2(M)>0$. So $M$ meets the conditions (a), (b), and (c). 

This finishes the proof of  Proposition \ref{0,k1}, which concludes the proof of Theorem \ref{0,k}.
\end{proof}

Finally we are ready to prove the Realization Theorem.

\begin{proof}[Proof of Theorem \ref{realization}]
For $n=3$, the theorem follows from Theorem \ref{t:CMV} (\cite{CMV}). For $n\geq 6$, the theorem follows from Theorem \ref{extension} (\cite{NSTWW}).  For $n=4,5$, by implementing Theorem \ref{0,k}, the proof follows by the same strategy as in \cite{CMV}. So we only give an outline of the proof here. 

For any finite set $A\subset \mathbb{Z}$ that contains $0$, by Proposition 2.2 of \cite{CMV}, there exist a finite sequence $\{B(i)\}_{i=1}^k$, where each $B(i)$ is a finite sequence of non-zero integers such that $$A=\bigcap_{i=1}^k S_{B(i)}.$$ Here $S_{B(i)}$ denotes the set of sums of subsequences of $B(i)$, which includes $0$, as the sum of the empty subsequence.

For $n=4,5$, we take a closed, oriented, aspherical $(n-1)$-manifold $N$ as in Proposition \ref{0,k1}, and take a non-torsion element $b\in H^2(N;\mathbb{Z})$. By Propositions \ref{appl2} and \ref{0,k1}, Theorem \ref{0,k} actually holds for these particular  $\tilde{N}_{lb}$ and $\tilde{N}_{klb}$, with $k \in \mathbb{Z}$ and $l\in \mathbb{Z}\setminus \{0\}$.

We take distinct prime numbers $p_1,..., p_k$ that are greater than absolute values of all numbers in all the sequences $B(i)$, with $i=1,...,k$. For each $i=1,...,k$, let $\alpha_i=p_i\cdot \prod_{\beta\in B(i)} \beta$. Then we use the above $N$ and $b$ to construct the following closed oriented $n$-manifolds: $$N_i=\tilde{N}_{\alpha_ib},\ M_i=\#_{\beta\in B(i)}\tilde{N}_{\frac{\alpha_i}{\beta}b},$$ where $N_i$ is aspherical. By Lemma 3.5 of \cite{NWW} and Theorem \ref{0,k}, we have 
$$D(M_i,N_i)=S_{B(i)}, \, \text{and} \, D(M_i,N_j)=\{0\}\,  \, \text{if} \,  i\ne j.$$

Then Proposition 3.7 of \cite{CMV} implies the existence of an integer $l\geq 0$, such that $$D\Big((\#_{i=1}^k M_i)\#(\#^l (S^{n-1}\times S^1)),\#_{i=1}^kN_i\Big)=\bigcap_{i=1}^kD(M_i,N_i)=\bigcap_{i=1}^kS_{B(i)}=A.$$
\end{proof}

\bibliographystyle{ams}

\end{document}